\documentclass[10pt,a4paper,openright,twoside]{article}
\usepackage[english]{babel}
\usepackage[latin1]{inputenc}
\usepackage{graphicx}
\usepackage{newlfont}
\usepackage{amssymb}
\usepackage{amsmath}
\usepackage{latexsym}
\usepackage{amsthm}
\usepackage{mathrsfs}
\usepackage{tikz-cd}
\usepackage{afterpage}
\usepackage{adjustbox}
\usepackage{tikz}
\usetikzlibrary{tikzmark}
\usetikzlibrary{positioning}
\usepackage{xcolor}
\usepackage{colortbl}
\usepackage{pdflscape}
\usepackage{rotating}
\usepackage{adjustbox}
\usepackage{float}

\theoremstyle{plain}                    
\newtheorem{teo}{Theorem}[section]
\newtheorem{prop}[teo]{Proposition}    
\newtheorem{cor}[teo]{Corollary}       
\newtheorem{lem}[teo]{Lemma} 
\theoremstyle{definition}               
\newtheorem{defin}{Definition}
      
\theoremstyle{remark}        
\newtheorem{rk}{Remark}[section]
      
\title{Stable cohomology of the moduli space of trigonal curves}
\author{Angelina Zheng\thanks{University of Pavia, Department of Mathematics
	``Felice Casorati", Pavia, Italy}}
\date{}
\begin{document}
\maketitle	
\begin{abstract}We prove that the rational cohomology $H^i(\mathcal{T}_g;\mathbf{Q})$ of the moduli space of trigonal curves of genus $g$ is independent of $g$ in degree $i<\lfloor g/4\rfloor.$ This makes possible to define the stable cohomology ring as $H^\bullet(\mathcal{T}_g;\mathbf{Q})$ for a sufficiently large $g.$  We also compute the stable cohomology ring, which turns out to be isomorphic to the tautological ring. This is done by studying the embedding of trigonal curves in Hirzebruch surfaces and using Gorinov-Vassiliev's method.
\end{abstract}
\section{Introduction and results}
Let us denote by $\mathcal{M}_g$ the moduli space of smooth complex projective curves of genus $g.$ 
The dependence on $g$ of the rational cohomology of $\mathcal{M}_g$ has been an active topic of research for a long time and it is now known, due to the work of Harer \cite{Har85}, and later of Ivanov \cite{Iva89} and Boldsen \cite{Bol12}, that the cohomology ring $H^i(\mathcal{M}_g;\mathbf{Q})$ is independent of the genus $g$ in the range $2g\geq 3i+2.$
Mumford also conjectured in \cite{Mum83} that the stable cohomology ring is generated by tautological classes and this conjecture was proved by Madsen and Weiss in \cite{MW07} using topological techniques.

In this work we investigate the behaviour of the rational cohomology ring of the moduli space $\mathcal{T}_g$ of complex trigonal curves of genus $g,$ i.e. the locus of smooth non-hyperelliptic curves admitting a $g^1_3$ inside the moduli space of curves. Our main result is that the rational cohomology ring of $\mathcal{T}_g$ behaves similarly to that of $\mathcal{M}_g$. We will compute its stable range and its stable cohomology ring, and prove that it is generated by tautological classes as well.

The moduli space $\mathcal{T}_g$ of trigonal curves is also strictly related to the Hurwitz scheme $\mathcal{H}_{3,g}$, parametrizing pairs $(C,\alpha)$, consisting of a smooth curve $C$ of genus $g$ and a degree 3 cover $\alpha:C\rightarrow\mathbf{P}^1$, up to isomorphism. In fact, for $g\geq 5,$ we have $\mathcal{H}_{3,g}\cong\mathcal{T}_g$ \cite[III.B-3.(i)]{ACGH85}.
Thus, we can think of $\mathcal{T}_g$ as a moduli space of covers of $\mathbf{P}^1$.

The topology of moduli spaces of covers of $\mathbf{P}^1$ has attracted great interest in the last few decades, and the stabilization of their cohomology has been extensively studied not only in algebraic geometry but also in arithmetic geometry and number theory. 
Specifically, Ellenberg, Vankatesh and Westerland \cite{EVW15} proved the stabilization for the homology of Hurwitz spaces, which are moduli spaces of $G\mbox{-}$covers of $\mathbf{P}^1$, i.e. Galois covers $\alpha:C\rightarrow\mathbf{P}^1$ with $G\cong Aut(\alpha)$, see \cite{RW06}. This result was motivated by the \emph{Cohen-Lenstra heuristics}, which they analyzed from the study of the asymptotic behavior of the number of points of Hurwitz schemes over a finite field $\mathbf{F}_q$ with $q$ points. 

Our description of the stable rational cohomology of $\mathcal{T}_g$ will be obtained by studying the loci of trigonal curves lying on each Hirzebruch surface, hence each stratum in the \emph{Maroni stratification}, \cite{Mar46}. We will use Gorinov-Vassiliev's method \cite{Vas}, \cite{Gor}, \cite{Tom05}, which reduces the computation of the cohomology of complements of discriminants to the study of a simplicial resolution of the discriminant, based on a classification of the singular loci of its elements. In particular we won't consider the whole classification, but only the families of singular configurations having low codimension in the vector space in which the discriminant is defined.

Our starting point will be the approach in \cite{Zhe}, where we computed the rational cohomology of the moduli space of trigonal curves of genus 5. For $g=5$, in fact, all trigonal curves lie on the first Hirzebruch surface $\mathbb{F}_1$ as smooth divisors. However, $\mathbb{F}_1$ and the other $\mathbb{F}_n$'s contain other trigonal curves of higher genera. For higher values of $g,$ the classification of the singular loci of such curves is more complicated, but we will see that the families of singular configurations we will consider have a description which is analogous to the one we had for $g=5.$ This will allow us to compute the cohomology of trigonal curves lying on $\mathbb{F}_n$ in a certain range.

We will exhibit the procedure for any Hirzebruch surface of degree $n\geq0$ in order to compute the stable cohomology of the locus of trigonal curves lying on it, defined as $N_n:=\{\left[C\right]\in\mathcal{T}_g|\, C\text{ has Maroni invariant } n\}.$  The result is described in the following proposition.
\begin{prop}\label{propMaroniStrata}
	The rational cohomology of $N_n$, for $n\geq1,$ in degree $i\leq\left\lfloor\frac{g-3n+2}{4}\right\rfloor,$ is 
		\begin{equation}\label{stratum}
			H^i(N_n;\mathbf{Q})=\begin{cases}
				\mathbf{Q},&i=0,\\
				\mathbf{Q}(-1)&i=2,\\
				\mathbf{Q}(-3)&i=5,\\
				\mathbf{Q}(-4)&i=7,\\
				0&\text{otherwise};
			\end{cases}
		\end{equation}
		the rational cohomology of $N_0$, in degree $i\leq\left\lfloor\frac{g+2}{4}\right\rfloor$, is 
		\begin{equation}\label{stratum0}
			{H}^i(N_0;\mathbf{Q})=\begin{cases}
				\mathbf{Q},&i=0,\\
				\mathbf{Q}(-3),&i=5,\\
				0,&\text{otherwise};
			\end{cases}
		\end{equation}
	where $\mathbf{Q}(-k)$ denotes the Tate Hodge structure of weight $2k$.
\end{prop}
\begin{rk}
By Deligne's Hodge theory cohomology and Borel-Moore homology of complex quasi-projective varieties carry mixed Hodge structures. In particular, we will work with mixed Hodge structures which are extensions of rational Tate Hodge structures.
\end{rk}

Let us anticipate that the description in Proposition \ref{propMaroniStrata} generalizes to non singular sections of $\mathcal{O}_{\mathbb{F}_n}(hE_n+dF_n)$ with $i\leq\left\lfloor\frac{g+(3-2h)n+2}{4}\right\rfloor$ for $h\geq3,$ where $E_n,F_n$ are the classes of the unique irreducible curve of negative self-intersection, when $n>0,$ and of a fiber of the ruling, respectively. For the details we refer to Remark \ref{generalization}.

By considering then the spectral sequence associated to the Maroni stratification of $\mathcal{T}_g,$ we finally obtain a description of the stable cohomology of $\mathcal{T}_g,$ for $g$ sufficiently large. Precisely,
\begin{teo}The rational cohomology of $\mathcal{T}_g,$ in degree $i<\left\lfloor\frac{g}{4}\right\rfloor$, is 
$$H^i(\mathcal{T}_g;\mathbf{Q})=\begin{cases}
\mathbf{Q}&i=0,\\
\mathbf{Q}(-1)&i=2,\\
\mathbf{Q}(-2)&i=4,\\
0&\text{otherwise}.
\end{cases}$$
\end{teo}

\begin{rk}
	Note that, if $g\equiv 2\operatorname{mod} 4$, the above description of the rational cohomology of $\mathcal{T}_g$ holds for $i\leq\left\lfloor\frac{g}{4}\right\rfloor.$
\end{rk}

In \cite{PV15} Patel and Vakil proved that the rational Chow ring $A^*_{\mathbf{Q}}(\mathcal{T}_g)$ coincides with its tautological subring, denoted $R^*_{\mathbf{Q}}(\mathcal{T}_g),$ which is defined as the subring generated by the pullback of tautological classes in $A^*_{\mathbf{Q}}(\mathcal{M}_g).$ In particular, they proved that it is generated by a single class in codimension 1, the kappa class $\kappa_1.$

Then, our main result also implies 
\begin{cor}
For $g,i$ such that $i<\left\lfloor\frac{g}{4}\right\rfloor$, 
$$\begin{cases}H^{i}(\mathcal{T}_g;\mathbf{Q})\cong R^{i/2}_{\mathbf{Q}}(\mathcal{T}_g),&i\text{ even,}\\
	H^{i}(\mathcal{T}_g;\mathbf{Q})=0,&i\text{ odd.}\\ 
	\end{cases}$$
\end{cor}
Where the isomorphism in the even degree case is induced by the cycle class map.
\begin{rk}
	For $g=3,4,5$, the cohomology $H^{\bullet}(\mathcal{T}_g;\mathbf{Q})$ is completely known from \cite{Loo}, \cite{Tom05}, \cite{Zhe}, respectively. However, in none of these cases the cohomology ring is tautological. Specifically, in these three cases, there are cohomology classes with Hodge weight different from the cohomological degree, which thus cannot be algebraic.
\end{rk}

From the proof of our main result, we can also deduce the stable cohomology of the moduli space $\mathcal{T}^\dagger_g$ of framed triple covers, i.e. the moduli space parametrizing pairs $(C,\alpha)$ with $C$ a smooth curve of genus $g$ and $\alpha$ a degree $3$ map from $C$ to a fixed $\mathbf{P}^1.$ Notice that $\mathcal{T}^\dagger_g$ is the underlying moduli space of the stack $\mathcal{H}_{3,g}^\dagger,$ defined in \cite{PV15}.
\begin{cor}
	Let $g\geq6,$ the rational cohomology of $\mathcal{T}_{g}^\dagger,$ in degree $i<\left\lfloor\frac{g}{4}\right\rfloor$, is 
	$$H^i(\mathcal{T}_{g}^\dagger;\mathbf{Q})=\begin{cases}
	\mathbf{Q}&i=0,\\
	\mathbf{Q}(-1)&i=2,\\
	\mathbf{Q}(-3)&i=5,\\
	\mathbf{Q}(-4)&i=7,\\
	0&\text{otherwise}.
\end{cases}$$
\end{cor}	
\begin{rk}
	Let us remark that our results prove that, for a sufficiently large $g,$ the rational Chow ring of $\mathcal{T}^\dagger_g$ is strictly smaller than that of $\mathcal{T}_g,$ which agrees with the description given by Canning and Larson in \cite{CL}. 
	Therefore, this contradicts \cite[Theorems A and B]{PV15}, which claim that the rational Chow ring of $\mathcal{T}^\dagger_g$ is isomorphic to that of $\mathcal{T}_g$. In fact,  there was an error in the last section of their preprint concerning relations between kappa classes, which has been corrected by Canning and Larson.
\end{rk}

\subsection*{Outline}
The paper is organized as follows. In section 2 we set up the notation and introduce results from algebraic topology and commutative algebra which will be used in the course of the main proofs.
Then, in section 3 we will present Gorinov-Vassiliev's method and apply it to our setting. Finally, in section 4 we give a proof of Proposition 1.1 and in section 5 we prove Theorem 1.2 and Corollary 1.4.

\subsection*{Acknowledgements} {The paper is a product of my PhD at University of Padova. I would like to thank first my advisor, Orsola Tommasi, for the helpful conversations and suggestions. I also thank Samir Canning and Hannah Larson for discussing with me their recent work and for sharing with me their description of the rational Chow ring of the Hurwitz stack $\mathscr{H}_{3,g}.$ In particular I thank Samir Canning for his help with the proof of Lemma 4.2.\\
I am also grateful to Melanie Matchett Wood for discussions on an early version of this work.}

\section{Notation and preliminary lemmas}
\subsection{Trigonal curves as divisors in $\mathbb{F}_n$}
Recall that, for $n\geq0,$ the $n\mbox{-}$th Hirzebruch surface is defined as $$\mathbb{F}_n=\mathbf{P}(\mathcal{O}_{\mathbf{P}^1}\oplus\mathcal{O}_{\mathbf{P}^1}(n)).$$
The Picard group and the intersection form on a Hirzebruch surface are well known.
\begin{prop}[{\cite[V.2]{Hart}}]\label{propHar}
	Let  $n\geq0,$ $\pi:\mathbb{F}_n\cong\mathbf{P}(\mathcal{O}_{\mathbf{P}^1}\oplus\mathcal{O}_{\mathbf{P}^1}(n))\rightarrow \mathbf{P}^1$ is a rational ruled surface and
	\begin{enumerate}
	\item $\operatorname{Pic}(\mathbb{F}_n)\cong\mathbf{Z}E_n\oplus\mathbf{Z}F_n,$ where $E_n$ is the class of the image of the section $(0,1)$ of $\mathcal{O}_{\mathbf{P}^1}\oplus\mathcal{O}_{\mathbf{P}^1}(n),$ which is the unique irreducible curve of negative self-intersection when $n>0,$ and $F_n$ is the class of any fiber of the ruling.
	\item $E_n, F_n$ satisfy $$E_n^2=-n,\qquad F_n^2=0,\qquad E_n\cdot F_n=1.$$
	\item $K_n \sim -2E_n + (-2-n)F_n$, where $K_n$ denotes the canonical divisor on $\mathbb{F}_n$.
	\end{enumerate}
\end{prop}
 \begin{rk}
 When $n=0,$ $\mathbb{F}_0\cong\mathbf{P}^1\times\mathbf{P}^1$ and in this case, $E_0$ $F_0$ are lines, each of a distinct ruling in $\mathbf{P}^1\times\mathbf{P}^1,$ both with trivial self-intersection.\end{rk}

It is also known that any trigonal curve of genus $g$ can be embedded in a Hirzebruch surface $\mathbb{F}_n$ and more precisely, by Proposition \ref{propHar}.3 and the genus formula, as a divisor of class
\begin{equation}\label{eq1}
	C\sim 3E_n+\frac{g+3n+2}{2}F_n,
\end{equation}
for a unique integer $n$ such that $g\equiv n\operatorname{mod}2$ and $0\leq n\leq (g+2)/3.$ The integer $n$ is called the Maroni invariant \cite{Mar46}. 

\begin{defin} Let $0\leq n\leq (g+2)/3$ be such that $g\equiv n\operatorname{mod}2$ and $d=\frac{g+3n+2}{2}\geq3n.$
Define $V_{d,n}$ to be the vector space of global sections of $\mathcal{O}_{\mathbb{F}_n}(3E_n+dF_n).$ Let $X_{d,n}$ be the open subset of sections defining smooth curves and define the discriminant locus $\Sigma_{d,n}$ as its complement in $V_{d,n}.$ 
\end{defin}

There is an explicit way to compute the dimension of $V_{d,n}.$\\
Recall that a further description of a Hirzebruch surface $\mathbb{F}_n,$ with $n\geq1,$ is given by blowing up the weighted projective space $\mathbf{P}(1,1,n)$ at its singular point $\left[0,0,1\right]$:
$$\mathbb{F}_n=Bl_{\left[0,0,1\right]}\mathbf{P}(1,1,n),$$
where $\mathbf{P}(1,1,n)=\operatorname{Proj}\mathbf{C}\left[x,y,z\right]$ such that $\operatorname{deg}x=\operatorname{deg}y=1$ and $\operatorname{deg}z=n.$
Then, a polynomial $f$ defining a trigonal curve of degree $d$ in $\mathbf{P}(1,1,n)$ is of the form:
\begin{equation}\label{poly}
	f(x,y,z)=\alpha_{d-3n}(x,y)z^3+\beta_{d-2n}(x,y)z^2+\gamma_{d-n}(x,y)z+\delta_d(x,y)=0
\end{equation}
where $\alpha_{d-3n}(x,y),\beta_{d-2n}(x,y),\gamma_{d-n}(x,y),\delta_d(x,y)$ are homogeneous polynomials in the coordinates $x,y,$ of degrees $d-3n$, $d-2n$, $d-n$, $d$ respectively, with $d\geq 3n$.\\
We can visualize the coefficients in the following figure:
\begin{equation}\label{par}
	\begin{array}{lcccccccccccccc}
		d-3n+1 &&		&        &        &\bullet&        &\dots&        &\dots&        &\bullet&         &         &\\
		d-2n+1 &&        &        &\bullet&         &\bullet&      &\dots&        &\bullet&         &\bullet&         &\\
		d-n+1    &&        &\bullet&        &\bullet&         &\dots&        &\dots&         &\bullet&        &\bullet&\\
		d+1&&\bullet&        &\bullet&         &\dots&        &\dots&         &\dots&        &\bullet&        &\bullet
	\end{array}
\end{equation}
where the $j\mbox{-}$th row, $j=1,\dots,4,$ corresponds to the coefficients of monomials $x^{a}y^{b}z^{3-j+1}$ with $a+b=d-n(3-j+1).$
Counting the number of parameters we get that $v_{d,n}:=\text{dim}_{\mathbf{C}}V_{d,n}=4d+4-6n,$ when $n\geq 1.$ Note that this also agrees when $n=0:$ $V_{d,0}$ is isomorphic to the vector space of polynomials of bidegree $(3,d),$ with $d=\frac{g+2}{2},$ on $\mathbf{P}^1\times\mathbf{P}^1,$ i.e.
$$V_{d,0}\cong \mathbf{C}\left[x_0,x_1,y_0,y_1\right]_{3,d}\cong \mathbf{C}^{4(d+1)}.$$
\subsection{Configuration spaces}
Configuration spaces play a crucial role in Gorinov-Vassiliev's method. In this subsection we give their definition together with some results that will be used repeatedly.
\begin{defin}
	Let $Z$ be a topological space, the space of ordered configurations of $k$ points in $Z$ is defined as
	$$F(Z,k)=Z^k\backslash\bigcup_{1\leq i<j\leq k}\{(z_1,\dots,z_k)\in Z^k|z_i=z_j\}.$$
	The quotient by the natural action of the symmetric group $\mathfrak{S}_k$ is denoted by $B(Z,k)$ and it is the space of unordered configurations of $k$ points in $Z.$\\
	For any subspace $Y\subseteq B(Z,k),$ the local system $\pm\mathbf{Q}$ over $Y$ is the one locally isomorphic to $\mathbf{Q}$ that changes its sign under any loop defining an odd permutation in a configuration from $Y.$ We will denote by $\bar{H}_{\bullet}(Y;\pm\mathbf{Q})$ the \emph{ Borel-Moore homology }of $Y$ with twisted coefficients, or the \emph{twisted Borel-Moore homology} of $Y.$ 
\end{defin}
Let us recall some basic results on the twisted Borel-Moore homology of some configuration spaces.
\begin{lem}[{\cite[Lemma 2]{Vas}}]\label{VAS}
	\begin{itemize}
	\item[a.] $\bar{H}_{\bullet}(B(\mathbf{C}^N,k);\pm\mathbf{Q})$ is trivial for any $N\geq1$, $k\geq2.$\sloppy
	\item[b.] $\bar{H}_{\bullet}(B(\mathbf{P}^N,k);\pm\mathbf{Q})=H_{\bullet-k(k-1)}(G(k,\mathbf{C}^{N+1});\mathbf{Q})$, where $G(k,\mathbf{C}^{N+1})$ is the Grassmann manifold of $k\mbox{-}$dimensional subspaces in $\mathbf{C}^{N+1}$. In particular the group $\bar{H}_{\bullet}(B(\mathbf{P}^N,k);\pm\mathbf{Q})$ is trivial if $k>N+1.$
\end{itemize}	
\end{lem}

The twisted Borel-Moore homology of a configuration space $B(Z,k)$ can immediately be read off from Lemma \ref{VAS} for any space $Z$ which admits a stratification whose strata are affine spaces. Such a stratification induces a stratification on $B(Z,k)$, whose strata record the number of points in each stratum of $X$.

\begin{lem}[{\cite[Lemma 2.3]{Zhe}}]
	$\bar{H}_{\bullet}(B(\mathbf{P}^2\backslash\{\text{point}\},2);\pm\mathbf{Q})$ is $\mathbf{Q}(3)$ in degree $6$ and $0$ in all other degrees. $\bar{H}_{\bullet}(B(\mathbf{P}^2\backslash\{\text{point}\},k);\pm\mathbf{Q})$ is trivial if $k\geq3$.
\end{lem}
Note that $\mathbb{F}_n$, for $n\geq 0,$ can be stratified into two affine cells, one isomorphic to $\mathbf{P}^2\backslash\{\text{point}\}\cong \mathbf{A}^2\cup\mathbf{A}^1$ and the other isomorphic to $\mathbf{P}^1\cong\mathbf{A}^1\cup\mathbf{A}^0.$ Then the twisted Borel-Moore homology of $B(\mathbb{F}_n,k)$ can be deduced from Lemmas 2.2 and 2.3. Precisely,

\begin{lem} For $n\geq 0,$
	\begin{equation*}
		\bar{H}_i(B(\mathbb{F}_n,1);\pm\mathbf{Q})=\begin{cases}
			\mathbf{Q},&i=0,\\
			2\mathbf{Q}(1),&i=2,\\
			\mathbf{Q}(2),&i=4,\\
			0,& otherwise;
		\end{cases}
	\end{equation*}
	\begin{equation*}
		\bar{H}_i(B(\mathbb{F}_n,2);\pm\mathbf{Q})=\begin{cases}
			2\mathbf{Q}(1),&i=2,\\
			2\mathbf{Q}(2),&i=4,\\
			2\mathbf{Q}(3),&i=6,\\
			0,& otherwise;
		\end{cases}
	\end{equation*}
	\begin{equation*}
		\bar{H}_i(B(\mathbb{F}_n,3);\pm\mathbf{Q})=\begin{cases}
			\mathbf{Q}(2),&i=4,\\
			2\mathbf{Q}(3),&i=6,\\
			\mathbf{Q}(4),&i=8,\\
			0,& otherwise;
		\end{cases}
	\end{equation*}
	\begin{equation*}
		\bar{H}_i(B(\mathbb{F}_n,4);\pm\mathbf{Q})=\begin{cases}
			\mathbf{Q}(4),&i=8,\\
			0,& otherwise;
		\end{cases}
	\end{equation*}
	\begin{equation*}
		\bar{H}_\bullet(B(\mathbb{F}_n,k);\pm\mathbf{Q})=0,\qquad \forall k\geq5. 
	\end{equation*}
\end{lem}

Note also that the twisted Borel-Moore homology groups computed above agree with the ones computed by Tommasi in \cite[Lemma 2.13]{Tom05} for $\mathbb{F}_0\cong\mathbf{P}^1\times\mathbf{P}^1.$ 
\subsection{Codimensions of spaces of sections in $V_{d,n}$}

We now want to consider elements in $V_{d,n}$ which are singular at the configuration spaces that we have just discussed. Since we are dealing with surfaces, requiring any polynomial $f\in V_{d,n}$ to be singular at any point in $\mathbb{F}_n$ will impose 3 linearly independent conditions. 
If we require $f$ to be singular at a configuration of $N$ points we expect the number of imposed conditions to be $3N$ and we will prove that this is indeed what happens when $d$ is sufficiently large with respect to $N.$ 

\begin{lem}\label{Codim}
	Fix $N\geq1$. For any $n\geq 0,$ the restriction of
	$$\{(f,p_1,\dots,p_N)\in V_{d,n}\times B(\mathbb{F}_n,N)|p_1,\dots,p_N\in\emph{Sing}(f)\}\xrightarrow{\pi} B(\mathbb{F}_n,N)$$
	to the locus where at most two points $p_i$ lie on the same line of the ruling is a vector bundle of rank $v_{d,n}-3N$ provided $d\geq2N+3n-1$ holds.
\end{lem}
\begin{rk}
	Before proving the lemma, note that we can further simplify the assumption that no more than two points lie on the same line of the ruling, by considering only the case where they all belong to distinct lines of the ruling. Clearly, curves which are singular at pairs of points in the same line of the ruling are easier to treat and we will indeed see that we can limit ourselves to consider curves of smaller degree which are singular at points lying on distinct lines of the ruling. The reason behind this is that, in both cases, the vector subspaces in $V_{d,n}$ of curves which are singular at these points have the expected codimension.
	
	In fact, let us consider first a set of $N=2k$ points consisting of $k$ pairs of points on $k$ distinct lines of the ruling. It is easy to show that the vector subspace of curves which are singular at these $2k$ points is non empty for $d\geq\frac{3}{2}(k+n)-1,$ which is always satisfied when $d\geq2N+3n-1$, and it is given exactly by all polynomials of the form
\begin{equation}\label{ell}\ell_1 \cdots\ell_k g,
\end{equation}
	where $\ell_1,\dots,\ell_k$ are the equations of the $k$ lines of the ruling containing the $2k$ points and $g$ is a section of $\mathcal{O}_{\mathbb{F}_n}(3E_n+(d-k)F_n),$ vanishing at the $2k$ prescribed points. Then, counting parameters as we have done in \eqref{par}, the vector subspace generated by these polynomials has dimension $4(d-k)+4-6n-2k$, which is non-negative by the assumption $d\geq\frac{3}{2}(k+n)-1.$ Hence it has codimension $6k=3N$ in $V_{d,n}.$ 
	
	This also holds if we consider a set of $N$ points consisting of $2k$ points, defined as above, together with $h$ points, each lying of a distinct line of the ruling, all different from the $k$ lines of the ruling containing the $2k$ points. In this case the vector subspace of curves which are singular at these $N$ points is given exactly by polynomials of the form \eqref{ell}, where we further require $g$ to be singular at the $h$ points. As we will prove below, this last assumption will increase the codimension by $3h$ and thus the codimension in $V_{d,n}$ of the vector subspace generated by these polynomials will be $6k+3h=3N.$ 
\end{rk}
\begin{proof} 
	The vector bundle $\pi:\{(f,p_1,\dots,p_N)\in V_{d,n}\times B(\mathbb{F}_n,N)|p_1,\dots,p_N\in\emph{Sing}(f)\}\rightarrow B(\mathbb{F}_n,N)$ is a vector sub-bundle of the trivial bundle $V_{d,n}\times B(\mathbb{F}_n,N)\rightarrow B(\mathbb{F}_n,N).$ Moreover, for any subspace $A\subset B(\mathbb{F}_n,N),$ the restriction of $\{(f,p_1,\dots,p_N)\in V_{d,n}\times B(\mathbb{F}_n,N)|p_1,\dots,p_N\in\emph{Sing}(f)\}$ to $A$ is still a vector bundle.
	Thus we only need to prove that all fibres have same dimension $v_{d,n}-3N.$\sloppy
	
	By the above remark, we will assume that all $p_i$'s lie on distinct lines of the ruling. 
	Following the proof of \cite[Lemma 4]{Tom20}, let us fix a set of $N$ distinct points $\{p_1,\dots,p_N\}\subset\mathbb{F}_n.$ 
	
	Assume first that $n\geq1,$ and consider the evaluation map 
	\begin{equation*}
		\mathbf{C}\left[x,y,z\right]_d\xrightarrow{ev}M_{3, N}(\mathbf{C})	
	\end{equation*}
	which assigns to each $f(x,y,z)=\alpha(x,y)z^3+\beta(x,y)z^2+\gamma(x,y)z+\delta(x,y)$ in the weighted polynomial ring $\mathbf{C}\left[x,y,z\right]_d$, with $\operatorname{deg}x=\operatorname{deg}y=1$ and $\operatorname{deg}z=n,$ a $3\times N$ matrix whose $i$-th column is defined by
	\begin{equation*}
		\begin{cases}\begin{pmatrix}
				\partial f/\partial x(p_i)\\
				\partial f/\partial y(p_i)\\
				\partial f/\partial z(p_i)\\
			\end{pmatrix}& \text{if } p_i\in \mathbb{F}_n\backslash E_n;\\
			\begin{pmatrix}
				\partial \alpha/\partial x_0(p_i)\\
				\partial \alpha/\partial y_0(p_i)\\
				\beta(p_i)
			\end{pmatrix}& \text{if } p_i\in E_n;\\		
		\end{cases}
	\end{equation*}
	where $x_0,y_0$ denote the coordinates in $E_n\cong\{(\left[0,0,1\right],[x_0,y_0])\in\{\left[0,0,1\right]\}\times\mathbf{P}^1\}\subset Bl_{\left[0,0,1\right]}\mathbf{P}(1,1,n)\cong\mathbb{F}_n.$ 
	The evaluation map is a linear map and its kernel is exactly the fiber of $\pi$ over $\{p_1,\dots,p_N\}.$ Therefore, in order to prove the lemma, it is sufficient to show that $ev$ is surjective for $d\geq 2N+3n-1.$
	
	Consider first the case where $p_1\in \mathbb{F}_n\backslash E_n.$ After an appropriate change of coordinates we may assume that $p_1=\left[1,0,0\right].$
	Fix a degree $r\geq 3n+1$ and consider the polynomials 
	$$\varphi_0=x^r\ell_2^2\cdots\ell_N^2,$$
	$$\varphi_1=x^{r-1}y\ell_2^2\cdots\ell_N^2,$$
	$$\varphi_2=x^{r-n}z\ell_2^2\cdots\ell_N^2.$$
	where $\ell_2,\dots,\ell_N$ are the equations of the lines of the ruling containing $p_2,\dots,p_N.$
	All $\varphi_i$ vanish with multiplicity 2 at $p_2,\dots,p_N$ and hence they are all sent to matrices with trivial entries outside the first column.
	Moreover, since $p_1\not\in\ell_i,$ for $i\geq2,$ then 
	$$\ell_2\cdots\ell_N(p_1)=a_0\neq0;\qquad\frac{\partial\ell_2\cdots\ell_N}{\partial x}(p_1)=Na_0;\qquad\frac{\partial\ell_2\cdots\ell_N}{\partial y}(p_1)=a_1\neq0.$$
	and the evaluations on such polynomials are
	$$ev(\varphi_0)=
	\begin{pmatrix}
		ra_0^2+2Na_0^2&0&\dots\\
		2a_1&0&\dots\\
		0&0&\dots\\
	\end{pmatrix},$$
	$$ev(\varphi_1)=\begin{pmatrix}
		0&0&\dots\\
		a_0^2&0&\dots\\
		0&0&\dots
	\end{pmatrix},$$
	$$ev(\varphi_2)=\begin{pmatrix}
		0&0&\dots\\
		0&0&\dots\\
		a_0^2&0&\dots\\
	\end{pmatrix}.$$
	Hence $ev(\varphi_0),ev(\varphi_1),ev(\varphi_2)$ are linearly independent generators for the subspace in $M_{3, N}(\mathbf{C})$ of matrices with zeros outside the first column.
	
	Next, consider the case where $p_1\in E_n:$ it is of the form $(\left[0,0,1\right],\left[x_0,y_0\right])\in Bl_{\left[0,0,1\right]}\mathbf{P}(1,1,n),$ and we may assume $p_1=(\left[0,0,1\right],\left[1,0\right]),$ after an appropriate change of coordinates. We now define 
	$$\varphi_0=x^{r-3n}z^3\ell_2^2\cdots\ell_N^2,$$
	$$\varphi_1=x^{r-3n-1}yz^3\ell_2^2\cdots\ell_N^2,$$
	$$\varphi_2=x^{r-2n}z^2\ell_2^2\cdots\ell_N^2,$$
	where $\ell_2,\dots,\ell_N$ are defined as in the previous case.
	The evaluations on these polynomials now are
	$$ev(\varphi_0)=
	\begin{pmatrix}
		(r+2N-3n)a_0^2&0&\dots\\
		2a_0a_1&0&\dots\\
		0&0&\dots\\
	\end{pmatrix},$$
	$$ev(\varphi_1)=\begin{pmatrix}
		0&0&\dots\\
		a_0^2&0&\dots\\
		0&0&\dots
	\end{pmatrix},$$
	$$ev(\varphi_2)=\begin{pmatrix}
		0&0&\dots\\
		0&0&\dots\\
		a_0^2&0&\dots\\
	\end{pmatrix},$$
	which are again linearly independent generators for the subspace in $M_{3, N}(\mathbf{C})$ of matrices with zeros outside the first column.
	
	Hence, we have proved that all matrices in $M_{3, N}(\mathbf{C})$ with trivial entries outside the first column belong to the image of $ev$ and by symmetry this can be generalized to all the other columns, proving the surjectivity of $ev.$ This has been proved using polynomials $\varphi_i$ which have degree $d\geq r+2(N-1)\geq2N+3n-1.$ This bound is actually sharp: if $E_n$ is a component of the curve, then we can refer to \eqref{par}: by counting the number of parameters of a polynomial $f$ defining a curve having $E_n$ as a component and $N$ distinct singular points on $E_n,$ we get that $d\geq2N+3n-1,$ as desired.\\
	
Finally, when $n=0,$ recall that $V_{d,0}\cong \mathbf{C}\left[x_0,x_1,y_0,y_1\right]_{3,d},$ hence any polynomial $f\in\mathbf{C}\left[x_0,x_1,y_0,y_1\right]_{3,d}$ is homogeneous of degree 3 in the variables $x_0,x_1$ and homogeneous of degree $d$ in the variables $y_0,y_1,$ at the same time. By definition, $f$ is singular at a point $p$ if and only if $\frac{\partial f}{\partial x_0}(p)=\frac{\partial f}{\partial x_1}(p)=\frac{\partial f}{\partial y_0}(p)=\frac{\partial f}{\partial y_1}(p)=0,$ and these are actually three independent conditions. In fact, if we denote $p=(\left[X_0,X_1\right],\left[Y_0,Y_1\right])$ the coordinates of $p\in\mathbf{P}^1\times\mathbf{P}^1,$ then $f(x_0,x_1,Y_0,Y_1)=g(x_0,x_1)\in \mathbf{C}\left[x_0,x_1\right]_3$ and $f(X_0,X_1,y_0,y_1)=h(y_0,y_1)\in \mathbf{C}\left[y_0,y_1\right]_d.$
By Euler's formula both on $g$ and $h$, the vanishing of any three partial derivatives forces also the fourth one to be zero.
So we can define the evaluation map as	
\begin{align*}
		\mathbf{C}\left[x_0,x_1,y_0,y_1\right]_{3,d}&\xrightarrow{ev}M_{3, N}(\mathbf{C})	\\
		f&\mapsto\begin{pmatrix}
		\partial f/\partial x_0(p_1)&\dots&\partial f/\partial x_0(p_N)\\
		\partial f/\partial x_1(p_1)&\dots&\partial f/\partial x_1(p_N)\\
		\partial f/\partial y_0(p_1)&\dots&\partial f/\partial y_0(p_N)\\
		\end{pmatrix}.
\end{align*}
Notice that, since $E_0$ is now a line of the ruling distinct from the one containing $F_0,$ there is no need to discuss if $p_1\in E_0$ or not. Choose coordinates such that $p_1=(\left[1,0\right],\left[1,0\right])\in\mathbf{P}^1\times\mathbf{P}^1$ and define
$$\varphi_0=x_0^r\ell_2^2\cdots\ell_N^2,\qquad\varphi_1=x_0^{r-1}x_1\ell_2^2\cdots\ell_N^2,\qquad\varphi_2=x_0^{r-1}y_0\ell_2^2\cdots\ell_N^2,$$
where $r\geq 1$ and $\ell_2,\dots,\ell_N$ are again the equations of the lines of a ruling containing $p_2,\dots,p_N.$  One can check that the evaluations on these polynomials are again linearly independent generators for the subspace in $M_{3, N}(\mathbf{C})$ of matrices with zeros outside the first column, hence $ev$ is again surjective and the polynomials $\varphi_i$ now have degree $d\geq r+2(N-1)\geq2N-1,$ which agrees with what we have proved for $n\geq1.$
\end{proof}

\begin{rk}\label{rkEW}
	Let us notice that the above lemma can be generalized to sections of $\mathcal{O}_{	\mathbb{F}_n}(hE+dF)$ for any $h\geq3.$ Specifically, for such values of $h,$ the above statement also holds if we replace $V_{d,n}$ with $H^0(\mathbb{F}_n;\mathcal{O}_{\mathbb{F}_n}(hE+dF))$ and the inequality  $d\geq2N+3n-1$ with $d\geq2N+hn-1$.	
	Indeed, for $n=0$ the claim follows from exactly the same argument of the proof Lemma \ref{Codim}, while for $n>0$ the claim can be proved after a slight modification. 
	If $p_1\in\mathbb{F}_n\backslash E_n,$ it suffices to define the polynomials $\varphi_i$ exactly as in the proof of Lemma \ref{Codim}, with $r\geq hn+1.$
	On the other hand, if $p_1$ belongs to the exceptional divisor, and it is of the form  $p_1=(\left[0,0,1\right],\left[1,0\right])\in Bl_{\left[0,0,1\right]}\mathbf{P}(1,1,n),$ we define 
	$$\varphi_0=x^{r-hn}z^h\ell_2^2\cdots\ell_N^2,$$
	$$\varphi_1=x^{r-hn-1}yz^h\ell_2^2\cdots\ell_N^2,$$
	$$\varphi_2=x^{r-(h-1)n}z^{h-1}\ell_2^2\cdots\ell_N^2.$$
	In both cases the polynomials $\varphi_i$ have degree $d\geq2N+hn-1,$ and it is easy to check that their images through the evaluation map are again linearly independent generators for the subspace in $M_{3,N}(\mathbf{C})$ of matrices with zero outside the first column.
\end{rk}

\section{Vassiliev's spectral sequence}
In this section, we will exhibit the method that allows us to compute the first columns of the Vassiliev's spectral sequence, hence the stable part of the cohomology of $X_{d,n}.$

First of all, recall that the rational cohomology of $X_{d,n}$ is equivalent to the Borel-Moore homology of $\Sigma_{d,n}$ by Alexander duality:
$$\tilde{H}^{\bullet}(V_{d,n}\backslash\Sigma_{d,n};\mathbf{Q})\cong\bar{H}_{2v_{d,n}-1-\bullet}(\Sigma_{d,n};\mathbf{Q})(-v_{d,n}).$$

Then, in order to compute the Borel-Moore homology of $\Sigma_{d,n}$ we apply Gorinov-Vassiliev's method, which consists in constructing a simplicial resolution of $\Sigma_{d,n},$ admitting a filtration such that the Borel-Moore homology of the strata define a spectral sequence converging to that of $\Sigma_{d,n}.$\\

Precisely, assume that $d\geq2N+3n-1,$ let $I\subset\underline{N}=\{1,\dots,N\}$ and consider the simplex $\Delta_{\bullet},$ where
$$\Delta_{I}=\{g:I\rightarrow\left[0,1\right]|\sum_{i\in I}g(i)=1\}.$$ Then we will have natural maps $e_{IJ}:\Delta_I\rightarrow\Delta_J$ given by extending $g\in\Delta_I$ to 0 on $J\backslash I.$ \\
We define the $\underline{N}\mbox{-}$cubical spaces $\{\chi_I\}_{I\subset\underline{N}},\{\overline{\chi}_I\}_{I\subset\underline{N}},$ 
where, if $I=\{i_1,\dots,i_r\}$ such that $i_j\neq N$ for any $j=1,\dots,r,$
$$\chi_I:=\{(f,y_{1},\dots,y_{r})\in V_{d,n}\times \prod_{j=1}^r B(\mathbb{F}_n,{i_j})|y_1\subset\dots\subset y_r\subset\text{Sing}(f)\},$$
$$\chi_{\emptyset}:=\Sigma_{d,n},\qquad \text{and} \qquad\chi_{I\cup\{N\}}:=\{(f,y_{1},\dots,y_{r})\in \chi_I|f\in \overline{\Sigma}_N\};$$
$$\overline{\chi}_I:=\{(f,y_{1},\dots,y_{r})\in V_{d,n}\times \prod_{j=1}^r \overline{B(\mathbb{F}_n,{i_j})}|y_1\subset\dots\subset y_r\subset\text{Sing}(f)\},$$
$$\overline{\chi}_{\emptyset}:=\Sigma_{d,n},\qquad\text{and}\qquad\overline{\chi}_{I\cup \{N\}}  :=\{(f,y_{1},\dots,y_{r})\in \overline{\chi}_I|f\in \overline{\Sigma}_N\};$$
where $\overline{\Sigma}_N$ denotes the Zariski closure of the locus in $\Sigma_{d,n}$ of polynomials defining curves with at least $N$ distinct singular points.
The natural forgetful maps $\varphi_{IJ}:\chi_J\rightarrow \chi_I$, $\overline{\varphi}_{IJ}:\overline{\chi}_J\rightarrow \overline{\chi}_I$ give $\chi_\bullet$ and $\overline{\chi}_\bullet$ the structure of cubical spaces over the index set $\underline{N}$. 
Consider then their geometric realisation, which is defined for $\chi_{\bullet}$ as the quotient
$$|\chi|=\left(\bigsqcup_{I\subset\{1,\dots,N\}}\chi_I\times\Delta_I\right)/\sim,$$
where $((f,Y),g)\sim((f',Y'),g')$ if and only if $(f,Y)=\varphi_{IJ}(f',Y')$ and $g'=e_{IJ}(g),$ and similarly for $\overline{\chi}_{\bullet}.$\\
We can then construct a surjective map 
$$\psi:|\overline{\chi}|\rightarrow |\chi|$$ as follows: let $((f,Y),g)\in\overline{\chi}_I\times\Delta_I,$ and let $\left[(f,Y),g\right]$ be its corresponding class in $|\overline{\chi}|.$ We define $\psi(\left[(f,Y),g\right])$ as the class in $|\chi|$ of the element $((f,Z),h)\in\chi_J\times\Delta_J$ where
$$J:=\{j\in\{1,\dots,N\}| y_i\in Y_j \text{ for some }i\in I\}=\{j(i)|i\in I\},$$
$$Z:=\prod_{j\in J}z_j; z_j=y_i \text{ for any $i$ s.t. $j(i)=j$},$$
$$h:J\rightarrow\left[0,1\right],\,h(j):=\sum_{i\in I| j(i)=j} g(i).$$
We consider the spaces $|\overline{\chi}_\bullet|$ with the quotient topology under the equivalence relation $\sim$ of the direct topology of the $\overline{\chi}_I,$ and on $|\chi_{\bullet}|$ the topology induced by $\psi.$\\
\begin{prop}[\cite{Gor}]
	The augmentation $|\chi_\bullet|\rightarrow \chi_{\emptyset}=\Sigma_{d,n}$ is a homotopy equivalence and induces an isomorphism on the Borel-Moore homology groups. 
\end{prop}
\begin{proof}
It is sufficient to prove that the augmentation map is proper with contractible fibers. Properness follows from the properness of the maps $\overline{\chi}_I\times\Delta_I\rightarrow \Sigma_{d,n}:$ the preimage of a compact coincides with the composition of such maps and the continuous map $\psi.$\\
Let $f\in\Sigma_{d,n},$ the fiber over $f$ is a cone with vertex the point $\left[(f,p_1,\dots,p_k),k=1\right]$ if $f$ has exactly $k<N$ singular points $p_1,\dots,p_k$, or $\left[(f,N=1)\right]$ if $f\in\overline{\Sigma}_N$, so the augmentation map is contractible in both cases. 
\end{proof}
\begin{rk}
From the theory of cubical spaces, \cite[Section 5.3]{PS08}, the isomorphism $\bar{H}_{\bullet}(|\chi_\bullet|;\mathbf{Q})\cong\bar{H}_{\bullet}(\Sigma_{d,n};\mathbf{Q})$ is also an isomorphism of mixed Hodge structures.
\end{rk}
As anticipated $|\chi_{\bullet}|$ admits an increasing filtration
$$\text{Fil}_i|\chi_\bullet|:=\text{Im}\left(|\chi_{\bullet}|_{\underline{i}}|\hookrightarrow|\chi_{\bullet}|\right),$$
where $|\chi_{\bullet}|_{\underline{i}}|$ is the geometric realization of the cubical spaces restricted to the index set $\underline{i}.$\\
Define locally closed subsets $$F_i:=\text{Fil}_i\backslash \text{Fil}_{i-1},$$ 
\begin{prop}[\cite{Tom05}]
	The filtration $\text{Fil}_i$ defines a spectral sequence that converges to the Borel-Moore homology of $\Sigma_{d,n}$, whose $E^1_{p,q}\mbox{-}$term is isomorphic to $\bar{H}_{p+q}(F_p;\mathbf{Q}).$
\end{prop}
This spectral sequence is called the Vassiliev spectral sequence.\\

Then by the above proposition, in order to compute the Borel-Moore homology of the discriminant $\Sigma_{d,n}$ we need to consider first the Borel-Moore homology of each $F_i.$ By construction, we have that 
$$F_i=\left(\bigsqcup_{I\subset\{1,\dots,N\};\,\operatorname{max}I=i}\chi_I\times\Delta_I\right)/\sim.$$
When $i<N,$ by \cite[Thm. 3, Lemma 1]{Gor} $F_i$ is a non-orientable simplicial bundle over $\chi_{\{i\}}$ with fiber isomorphic to the interior of a $(i-1)\mbox{-}$dimensional simplex.\\
Moreover, by Lemma 2.5, since we are working under the assumption that $d\geq2N+3n-1,$
$\chi_{\{i\}}$ is a complex vector bundle over $B(\mathbb{F}_n,i)$ of rank $v_{d,n}-3i.$ Putting all together we obtain an explicit formula for the Borel-Moore homology of $F_i$ for $i<N,$ namely

\begin{equation}\label{BM}\bar{H}_{\bullet}(F_i;\mathbf{Q})=\bar{H}_{\bullet-2v_{d,n}+6i-i+1}(B(\mathbb{F}_n,i);\pm\mathbf{Q})\otimes\mathbf{Q}(v_{d,n}-3i),\end{equation}
Which can be computed by Lemma 2.4. In particular, for any $n\geq0,$ the configuration spaces $B(\mathbb{F}_n,k)$ have all the same twisted Borel-Moore homology, which is trivial for $k> 4.$ Thus, among the first $N-1$ strata, only $F_1,\dots,F_4$ contribute non-trivially to the Borel-Moore homology of $\Sigma_{d,n}$. They correspond to the following classification of singular configurations:
\begin{itemize}
	\item[(1)] One point, $\left[3\right]$ ;
	\item[(2)] Two points, $\left[6\right]$;
	\item[(3)] Three points, $\left[9\right]$;
	\item[(4)] Four points, $\left[12\right].$
\end{itemize}
By the formula \eqref{BM} and by Lemma 2.4, we compute the Borel-Moore homology of the associated strata and we get that the first columns of the $E^1\mbox{-}$page of the spectral sequence will look as the ones represented in Table \ref{table1}.
\begin{table}[h!]\caption{Spectral sequence converging to $\bar{H}_\bullet({\Sigma_{d,n};\mathbf{Q}})$.}\centering\label{table1}	
	\begin{tabular}{ c|ccccc }
		$2v_{d,n}-3$&$\mathbf{Q}(v_{d,n}-1)$ &&&&\\
		$2v_{d,n}-4$& & && &\\
		$2v_{d,n}-5$&$\mathbf{Q}(v_{d,n}-2)^2$ & && &\\
		$2v_{d,n}-6$& & && &\\
		$2v_{d,n}-7$&$\mathbf{Q}(v_{d,n}-3)$ &$\mathbf{Q}(v_{d,n}-3)^2$ && &\\
		$2v_{d,n}-8$& & && &  \\
		$2v_{d,n}-9$& &$\mathbf{Q}(v_{d,n}-4)^2$ && & \\
		$2v_{d,n}-10$& & && &\\
		$2v_{d,n}-11$& &$\mathbf{Q}(v_{d,n}-5)^2$ &$\mathbf{Q}(v_{d,n}-5)$& & \\
		$2v_{d,n}-12$& & && & \\
		$2v_{d,n}-13$& & &$\mathbf{Q}(v_{d,n}-6)^2$& & \\
		$2v_{d,n}-14$& & && &\\
		$2v_{d,n}-15$& & &$\mathbf{Q}(v_{d,n}-7)$& & \\
		$2v_{d,n}-16$& & && & \\
		$2v_{d,n}-17$& & &&$\mathbf{Q}(v_{d,n}-8)$ & \\
		$2v_{d,n}-18$& & && &\\

		\hline
		&1&2&3&4&$\dots$
	\end{tabular}
\end{table}
\begin{rk}
Notice that this does not agree with the first 5 columns of the spectral sequence obtained in \cite[Table 3]{Zhe}. Indeed, for $g=5$, we have that the inequality $d\geq2N+2$ is not satisfied when $N\geq 2,$ hence the corresponding configurations do not have the expected codimension.
\end{rk}
\begin{rk}
For $n=0,$ the spectral sequence agrees with the first 4 columns of the spectral sequence in \cite[Table 3]{Tom05}, twisted by $\mathbf{Q}(v_{d,0}-16)$ in degree $2(v_{d,0}-16)$.
\end{rk}
When $i=N$, recall that 
$$F_N=\left(\bigsqcup_{I\subset\{1,\dots,N\};\,\operatorname{max}I=N}\chi_I\times\Delta_I\right)/\sim.$$
Following \cite[Lemma 18]{Tom14} we can further stratify the stratum $F_N$ as the union of locally closed substrata
$$\phi_0=\left(\chi_{\{N\}}\times\Delta_{\{N\}}\right)/\sim,\qquad\phi_l=\left(\chi_{I\cup\{N\}}\times\Delta_{I\cup\{N\}}\right)/\sim;$$
where $I\subset\left\{1,\dots,N\right\}$ such that $\operatorname{max}I=l$ for any $1\leq l\leq N-1.$

Then, for any of these substratum we have natural maps
$$\phi_0\rightarrow \chi_{\{N\}}, \qquad\phi_l\rightarrow\chi_{\{l,N\}},$$
where $\phi_0\cong\chi_{\{N\}}$ by definition of $\sim,$ and the fiber of $\phi_l\rightarrow\chi_{\{l,N\}}$ is the interior of a $l\mbox{-}$dimensional simplex: it is a cone over the fiber of $F_l\rightarrow \chi_{\{l\}},$ which by \cite[Thm 3. Lemma 1]{Gor} is again the interior of a $(l-1)\mbox{-}$dimensional simplex.\\
Moreover, for any $f\in\overline{\Sigma}_N,$ the projections $(f,p_1,\dots,p_N)\rightarrow f$, $(f,p_1,\dots,p_N)\mapsto (f,p_1,\dots,p_l)$ 
define surjections
$$\{(f,p_1,\dots,p_N)\in V_{d,n}\times B(\mathbb{F}_n,N)|p_1,\dots,p_N\in\emph{Sing}(f)\}\rightarrow\chi_{\{N\}},$$ 
$$\{(f,p_1,\dots,p_N)\in V_{d,n}\times B(\mathbb{F}_n,N)|p_1,\dots,p_N\in\emph{Sing}(f)\}\rightarrow\chi_{\{l,N\}},$$
where the domain, by the assumption that $d\geq2N+3n-1$ and by Lemma 2.5, is a vector bundle of rank $v_{d,n}-3N$ over $B(\mathbb{F}_n,N),$ which has dimension $2N.$ Therefore, we have that 
$$\operatorname{dim}_{\mathbf{R}}\phi_0\leq 2v_{d,n}-2N,\qquad\text{and}\qquad\operatorname{dim}_{\mathbf{R}}\phi_l\leq 2v_{d,n}-2N+l;\qquad1\leq l\leq N-1.$$
Then, since the largest dimensional stratum is $\phi_{N-1},$ $\operatorname{dim}_{\mathbf{R}}F_N=\operatorname{dim}_{\mathbf{R}}\phi_{N-1}\leq 2v_{d,n}-N-1,$ so the Borel-Moore homology of $F_N$ must vanish in degree $k\geq 2v_{d,n}-N.$
As a consequence, when considering the whole spectral sequence, we have that, for $k\geq 2v_{d,n}-N,$ the Borel-Moore homology of $\Sigma_{d,n}$ is defined only by the strata $F_1,\dots F_4.$ By Alexander duality, this means that the cohomology of $X_{d,n}$ is given by that of the strata $F_1,\dots F_4,$ for $k< N\leq\frac{d-3n+1}{2}.$ 
\begin{rk}\label{generalization}
	Let us consider again the space of sections in $\mathcal{O}_{\mathbb{F}_n}(hE+dF)$ with $h\geq3$ and repeat the whole argument in this section in order to compute the Borel-Moore homology of the subset of singular sections in $H^0(\mathbb{F}_n;\mathcal{O}_{\mathbb{F}_n}(hE+dF)).$ The only differences that we may get from the $h=3$ case, when studying the Vassiliev spectral sequence, are determined by configuration spaces having at least $3$ points on the same line of the ruling. 
	However these configurations have all trivial twisted Borel-Moore homology by Lemma \ref{VAS}. Hence the first columns of the corresponding Vassiliev spectral sequence are exactly the same ones as in the spectral sequence converging to the Borel-Moore homology of $\Sigma_{d,n},$ in Table \ref{table1}. 
	Thus, the only difference is given by the stable range, which will now depend on $h$ and can be deduced from Remark \ref{rkEW}. The stable range in cohomology is indeed $i\leq\left\lfloor\frac{g+(3-2h)n+2}{4}\right\rfloor.$
	This agrees with \cite[Theorem 9.5]{EW15}, in which Erman and Wood proved that the probability of a curve of bidegree $(h,d)$ on a fixed Hirzebruch surface to be smooth is independent of $h$ for $h\geq3$ and $d\rightarrow \infty.$
\end{rk}

\section{Group action on $\mathbb{F}_n$}
In this section we compute the rational cohomology of each stratum of the Maroni stratification, by considering the action of the automorphism group of $\mathbb{F}_n$ for any $n\geq0,$ thus proving Proposition 1.1.

Let $G_n$ denote the automorphism group of $\mathbb{F}_n.$ \sloppy
When  $n\geq1,$ $G_n\cong Aut(\mathbf{P}(1,1,n))$, which is the group of automorphisms of the weighted graded ring $\mathbf{C}\left[x,y,z\right]$ with $\operatorname{deg}x=\operatorname{deg}y=1$ and $\operatorname{deg}{z}=n$, fixing the singular point $\left[0,0,1\right]$. Such automorphisms are of the form
	\begin{equation*}
	\begin{cases}
	x\mapsto a_1x+a_2y\\
	y\mapsto b_1x+b_2y\\
	z\mapsto cz+ q(x,y),
	\end{cases}
	\end{equation*} 
	where $a_1,a_2,b_1,b_2,c\in\mathbf{C}$ are such that $c(a_1b_2-a_2b_1)\neq0$ and $q\in\mathbf{C}\left[x,y\right]_n.$ Note that $\mathbf{C}\left[x,y\right]_n\cong\mathbf{C}^{n+1}$ is contractible, therefore $G_n$ is homotopy equivalent to the reductive group $\mathbf{C}^*\times GL_2(\mathbf{C}).$
	
Let us consider the quotient stack $\left[X_{d,n}/(\mathbf{C}^*\times GL_2(\mathbf{C}))\right]$ and denote by $X_{d,n}/(\mathbf{C}^*\times GL_2(\mathbf{C}))$ its underlying coarse moduli space.
By abuse of notation, since a quotient stack and its underlying coarse moduli space
have the same rational cohomology, we will drop the brackets.

The moduli space $X_{d,n}/(\mathbf{C}^*\times GL_2(\mathbf{C}))$ is the space of isomorphism classes of triples $(C,L,H)$ where $C$ is a trigonal curve of genus $g=2d-3n-2,$ $L$ is the linear system defining its trigonal structure and $H$ is a hyperplane section.
Thus $X_{d,n}/(\mathbf{C}^*\times GL_2(\mathbf{C}))$ is an orbifold $\mathbf{C}^{n+1}\mbox{-}$bundle over 
$N_n=\{\left[C\right]\in\mathcal{T}_g|\, C\text{ has Maroni invariant } n\}$ and we can deduce the stable rational cohomology of the latter from that of $X_{d,n}/(\mathbf{C}^*\times GL_2(\mathbf{C})).$\\

Before giving a proof of Proposition 1.1, let us observe that, when $n\geq 1,$ a generalized version of the Leray-Hirsch theorem can be applied to $H^{\bullet}(X_{d,n};\mathbf{Q})$ in order to recover the rational cohomology of  $X_{d,n}/ GL_2(\mathbf{C})$ from that of $X_{d,n}$ and of $GL_2(\mathbf{C})$.
\begin{prop}\label{divisibility}
For $n\geq 1,$ there is an isomorphism of graded $\mathbf{Q}\mbox{-}$vector spaces with mixed Hodge structures 
$$H^{\bullet}(X_{d,n}/ GL_2(\mathbf{C});\mathbf{Q})\otimes H^{\bullet}(GL_2(\mathbf{C});\mathbf{Q})\cong H^{\bullet}(X_{d,n};\mathbf{Q}).$$
\end{prop}
\begin{proof}
By \cite[Thm. 3]{PS} it is sufficient to prove the surjectivity of the map on cohomology
\[
\resizebox{\columnwidth}{!}{
$
	\rho^*:\bar{H}_{2v_{d,n}-i-1}(\Sigma_{d,n};\mathbf{Q})\cong H^i(X_{d,n};\mathbf{Q})\rightarrow H^i(GL_2(\mathbf{C});\mathbf{Q})\cong \bar{H}_{2\operatorname{dim}M-i-1}(D;\mathbf{Q}),
$}
\]
induced by the orbit map $\rho: GL_2(\mathbf{C})\rightarrow X_{d,n},$ where $M$ is the space of $2\times2$ matrices and $D$ the complement of $GL_2(\mathbf{C})$ in $M.$\\
The generators of $\bar{H}_\bullet(D;\mathbf{Q})$ are $\left[D\right]$ in degree $6$ and $\left[R\right]$ in degree $4,$ where $R\subset D$ is the subvariety of matrices with zeros in the first column. From \cite[Section 3.1]{Zhe} we already know that, for $n=1,$ $\rho^*$ is surjective: the preimages of the generators $\left[D\right]$ and $\left[R\right]$ are a non-zero multiple of the class $\left[\Sigma_{d,n}\right]\in\bar{H}_{2v_{d,1}-2}(\Sigma_{d,1};\mathbf{Q})$ and a non-trivial linear combination of the classes $\left[\Sigma_{d,n}^1\right], \left[\Sigma_{d,n}^2\right]\in\bar{H}_{2v_{d,1}-4}(\Sigma_{d,1};\mathbf{Q}),$ respectively, where $\Sigma_{d,n}^1$ is the subspace of polynomials in $V_{d,n}$ which are singular at a point on $E_n$ and $\Sigma_{d,n}^2$ is the subspace of polynomials in $V_{d,n}$ which are singular at a point on a line of a ruling. From Table 1, we observe that the class in degree $2v_{d,n}-2$ and the two classes in degree $2v_{d,n}-4$ appear in each Vassiliev's spectral sequence, with $n\geq 1,$ therefore $\rho^*$ must be surjective for any $n\geq1.$ In fact, by recalling that elements of $V_{d,n}$ are polynomials of the form \eqref{poly}, if we consider the extension of the orbit map $D\rightarrow\Sigma_{d,n},$ $f\in X_{d,n}$ and the subvariety $R$, the latter acts by
$$\begin{pmatrix}0&c_1\\0&c_2\end{pmatrix}\cdot f(x,y,z)=C(c_1,c_2)y^{d-3n}g(y,z),$$ where $C(c_1,c_2)\in\mathbf{C}$ and $g$ is a weighted polynomial in $\mathbf{C}\left[y,z\right]_{3n},$ with $\operatorname{deg}y=1,$ $\operatorname{deg}z=n.$ Thus, elements in $R$ are sent to polynomials defining curves which are union of the line of the ruling of equation $y=0,$ with multiplicity $d-3n,$ and some other curve of lower degree passing through a point of this line of the ruling. Similarly elements in $D$ will be sent to polynomials defining curves that are union of some line of the ruling, with multiplicity $d-3n,$ and another curve meeting this line of the ruling at some point. Therefore the preimages of $\left[D\right]$ and $\left[R\right]$ through $\rho^*$ must be a non-zero multiple of $\left[\Sigma_{d,n}\right]$ and a non-trivial linear combination of  $\left[\Sigma_{d,n}^1\right], \left[\Sigma_{d,n}^2\right]$, respectively, as predicted. 
\end{proof}

With the above result, we are now able to compute the rational cohomology of $N_n,$ in degree $i\leq\lfloor\frac{d-3n}{2}\rfloor.$
\begin{proof}[Proof of Proposition 1.1]
Consider first the case $n\geq 1.$ By applying Proposition \ref{divisibility}, the rational cohomology of the quotient $X_{d,n}/GL_2(\mathbf{C})$, in degree $i\leq\left\lfloor\frac{d-3n}{2}\right\rfloor$, will be 
\begin{equation}\label{GL}
H^i(X_{d,n}/GL_2(\mathbf{C});\mathbf{Q})=\begin{cases}
\mathbf{Q}, & i=0;\\
\mathbf{Q}(-2), & i=3;\\
\mathbf{Q}(-3), & i=5;\\
\mathbf{Q}(-5), & i=8;\\
0, & \text{otherwise}.
\end{cases}
\end{equation}
Consider then the spectral sequence associated to the bundle
$$X_{d,n}/GL_2(\mathbf{C})\xrightarrow{\mathbf{C}^{*}}X_{d,n}/(\mathbf{C^*}\times GL_2(\mathbf{C})),$$
which will look either as

\begin{table}[H]\caption{Spectral sequence converging to $H^\bullet(X_{d,n}/GL_2(\mathbf{C});\mathbf{Q}).$}\centering\label{tableBundle}	
		\begin{tabular}{ c|cccccccccc} 
			$1$&$\mathbf{Q}(-1)$\tikzmark{w} &&$\mathbf{Q}(-2)$&&&$\mathbf{Q}(-4)$\tikzmark{v}&&$\mathbf{Q}(-5)$&&\\
			$0$&$\mathbf{Q}$ &&\tikzmark{ww}$\mathbf{Q}(-1)$&&&$\mathbf{Q}(-3)$&&\tikzmark{vv}$\mathbf{Q}(-4)$&&\\
			\hline
			&$0$&$1$&$2$&$3$&$4$&$5$&$6$&$7$&$8$&$\dots$
		\end{tabular}
		\begin{tikzpicture}[overlay, remember picture, yshift=.25\baselineskip, shorten >=.5pt, shorten <=.5pt]
		\draw [shorten >=.1cm,shorten <=.1cm,->]([yshift=5pt]{pic cs:w}) -- ([yshift=5pt]{pic cs:ww});
		\draw [shorten >=.1cm,shorten <=.1cm,->]([yshift=5pt]{pic cs:v}) -- ([yshift=5pt]{pic cs:vv});
		\end{tikzpicture}
\end{table}
\noindent or as
\begin{table}[H]\caption{Alternative spectral sequence converging to $H^\bullet(X_{d,n}/GL_2(\mathbf{C});\mathbf{Q}).$}\centering\label{alternative}
	\begin{tabular}{ c|cccccccccc} 
		$1$&$\mathbf{Q}(-1)$\tikzmark{1a} &&$\mathbf{Q}(-2)$\tikzmark{4a}&$ \mathbf{Q}(-3) $\tikzmark{3a}&$\mathbf{Q}(-3)$&$\mathbf{Q}(-4)$\tikzmark{2a}&&$\mathbf{Q}(-5)$&&\\
		$0$&$\mathbf{Q}$ &&\tikzmark{11a}$\mathbf{Q}(-1)$&$\mathbf{Q}(-2)$&\tikzmark{44a}$\mathbf{Q}(-2)$&\tikzmark{33a}$\mathbf{Q}(-3)$&&\tikzmark{22a}$\mathbf{Q}(-4)$&&\\
		\hline
		&$0$&$1$&$2$&$3$&$4$&$5$&$6$&$7$&$8$&$\dots$
	\end{tabular}
	\begin{tikzpicture}[overlay, remember picture, yshift=.25\baselineskip, shorten >=.5pt, shorten <=.5pt]
		\draw [shorten >=.1cm,shorten <=.1cm,->]([yshift=5pt]{pic cs:1a}) -- ([yshift=5pt]{pic cs:11a});
		\draw [shorten >=.1cm,shorten <=.1cm,->]([yshift=5pt]{pic cs:2a}) -- ([yshift=5pt]{pic cs:22a});
		\draw [shorten >=.1cm,shorten <=.1cm,->]([yshift=5pt]{pic cs:3a}) -- ([yshift=5pt]{pic cs:33a});
		\draw [shorten >=.1cm,shorten <=.1cm,->]([yshift=5pt]{pic cs:4a}) -- ([yshift=5pt]{pic cs:44a});
	\end{tikzpicture}
\end{table}

\noindent
where in both spectral sequences the differentials must be non-trivial because of \eqref{GL}. \\
Both spectral sequences also imply that $H^2(N_n;\mathbf{Q})$ is generated by the Euler class $\xi$ of the $\mathbf{C}^*\mbox{-}$bundle, which is a non-zero multiple of $\kappa_1,$ by \cite{PV15}.
Moreover, the spectral sequence represented in Table \ref{alternative} would imply that $H^4(N_n;\mathbf{Q})$ is generated by $\kappa_1\cdot\xi=\alpha\kappa_1^2\neq 0.$

Hence, the choice of the spectral sequence corresponds to verify if $\kappa_1^2=0$ holds in $H^4(N_n;\mathbf{Q})$ and is due to the following lemma.

\begin{lem}
The relation $\kappa_1^2=0$ holds in $H^4(N_n;\mathbf{Q}).$
\end{lem}
The proof of the above result is due to Samir Canning. We refer to the end of the section for the details.

The spectral sequence associated to the orbifold $\mathbf{C}^*\mbox{-}$bundle  $X_{d,n}/GL_2\rightarrow X_{d,n}/(\mathbf{C^*}\times GL_2)$ is thus the one represented in Table \ref{tableBundle} and, for $n\geq 1,$ the rational cohomology of $N_n$ is the one described in \eqref{stratum}. Here, notice that the stable range $i\leq\left\lfloor\frac{d-3n}{2}\right\rfloor$ is the same obtained for $H^i(X_{d,n}/GL_2(\mathbf{C});\mathbf{Q}).$ Indeed, if we had a non-trivial class in $E_2^{\left\lfloor\frac{d-3n}{2}\right\rfloor,0}$ in Table 2, then additional non-trivial classes would also appear in $H^i(X_{d,n}/GL_2(\mathbf{C});\mathbf{Q})$, for $i\leq\left\lfloor\frac{d-3n}{2}\right\rfloor.$

Finally, when $n=0,$ the group $G_0$ acting on $\mathbb{F}_0$ is different from those we have considered when $n\geq1$. However, also in this case a generalized version of Leray-Hirsch theorem can be applied. From \cite[Section 3.1]{Tom05}, $G_0$ is a reductive group which is isogenous to $\mathbf{C}^*\times SL_2(\mathbf{C})\times SL_2(\mathbf{C}),$ whose cohomology is known.

Thus, the rational cohomology of $X_{d,0}/G_0$ has already been computed in \cite[Section 3.7]{Tom05} and by applying the generalized version of Leray-Hirsch theorem we get \eqref{stratum0}.
\sloppy
\end{proof}
What is left to do is to show that the spectral sequence associated to the orbifold $\mathbf{C}^*\mbox{-}$bundle is indeed the one represented in Table 2.
\begin{proof}[Proof of Lemma 4.2]
The proof of this Lemma heavily relies on the relations in the rational Chow ring of $\mathcal{T}_g$ computed in \cite[Section 4]{CL}.

The same constructions and machinery that Canning and Larson produced to find relations on the Chow ring of $\mathcal{T}_g$ can be adapted to our setting to find relations in $A^*(N_n).$

More precisely, given a trigonal curve $C$ of Maroni invariant $n$ and genus $g$ with associated degree $3$ map $\alpha:C\rightarrow\mathbf{P}^1,$ we can associate it to a rank 2, degree $g+3$ vector bundle $\mathcal{E}$, denoted as the dual of the Tschirnhausen module, \cite{Mir}. 

In fact, we know that $C$ canonically embeds in $\mathbb{F}_n:=\mathbf{P}(\mathcal{O}_{\mathbf{P}^1}\oplus \mathcal{O}_{\mathbf{P}^1}(n))$ that is isomorphic to $\mathbf{P}(\mathcal{E})$ where $\mathcal{E}:=\mathcal{O}_{\mathbf{P}^1}(a)\oplus \mathcal{O}_{\mathbf{P}^1}(b))$ is said to be of \emph{splitting type} $(a,b)$ with $a, b\in\mathbf{Z}$ such that $a\leq b,$ $b-a=n$ and $a+b=g+2$. 

This defines a map $\hat{\alpha}: \mathcal{T}_g\rightarrow \mathcal{B}_{3,g}$, where $\mathcal{B}_{3,g}$ is the moduli stack of rank 2, degree $g+3$ globally generated vector bundles on $\mathbf{P}^1\mbox{-}$bundles.

Following \cite[Section 5]{BV12} and \cite[Section 4]{CL2}, the locally closed substack in $\mathcal{B}_{3,g}$ corresponding to a fixed splitting type $(a,b)$ is isomorphic to the classifying stack $\mathcal{B}\mathcal{G}\times \mathcal{B}SL_2$ with $\mathcal{G}:=Aut(\mathcal{E}),$ which is homotopy equivalent to $\mathbf{C}^*\times\mathbf{C}^*.$

Thus, in order to get relations on $A^*(N_n),$ it suffices to replace the space denoted by $\mathcal{B}_{3,g}'$ in \cite[Section 4]{CL} with $\mathcal{B}\mathcal{G}\times \mathcal{B}SL_2.$ This proves that $A^*(N_n)$ is a quotient of $\mathbf{Q}\left[n_1,m_1,c_2\right]$ where $n_1$ and $m_1$ are respectively the first Chern classes of the line bundles $\mathcal{N}$ and $\mathcal{M},$ associated to each factor $\mathbf{C}^*$ in $\mathcal{G}$ and $c_2$ is the pullback of the universal second Chern class on $\mathcal{B}SL_2.$

In particular we obtain that $A^*(N_n)\cong\mathbf{Q}\left[n_1,m_1,c_2\right]/I,$ for the ideal 
\begin{footnotesize}
\begin{align*} 
	I=&((-9b+8g+12)n_1+(9b-g-6)m_1,\\ &4n_1^2-n_1m_1+4m_1^2+(-9b^2+9bg-4g^2	+18b-12g-8)c_2,\\
	&(-12b+12g+20)n_1^2 
	-2n_1m_1+(12b-4)m_1^2+\\
	&(-12b^2
	g+12bg^2-4g^3-18b^2
	+42bg-20g^2+36b-32g-16)c_2,\\
	&4n_1^3
	+4m_1^3+(-12b^2+24bg-12g^2
	+42b-40g-32)n_1c_2+(-12b^2
	+6b+2g+4)m_1c_2),
\end{align*}\end{footnotesize}
which can be computed using the code written by Canning and Larson, which is provided at \cite{code}, adapted to our case.

The output of the computation shows a relation between $n_1$ and $m_1$ in codimension 1 and two linear independent relations between $n_1^2$ and $c_2$ in codimension 2, thus proving that $A^2(N_n)=0.$
\end{proof}

\section{Maroni stratification}
Recall that $\mathcal{T}_g$ has a natural stratification by the Maroni invariant:
	\begin{equation}\label{Mar}\begin{cases}
			\mathcal{N}_s\subset\dots\subset\mathcal{N}_{0}=\mathcal{T}_g,&\text{if $g$ is even},\\
			\mathcal{N}_s\subset\dots\subset\mathcal{N}_{1}=\mathcal{T}_g,&\text{if $g$ is odd},\\
		\end{cases}
	\end{equation}\sloppy
	where $s$ is the largest index with the same parity as $g$ satisfying $s\leq \left\lfloor\frac{g+2}{3}\right\rfloor$
	and for all $0\leq n\leq s,$ $g\equiv n(\operatorname{mod}2)$ we denote by $\mathcal{N}_n$ the closed subscheme
	\begin{equation*}
		\mathcal{N}_n:=\{\left[C\right]\in\mathcal{T}_g|\, C\text{ has Maroni invariant }\geq n\}\subseteq \mathcal{T}_g.
\end{equation*}
Notice that $N_n=\mathcal{N}_n\backslash\mathcal{N}_{n+2},$ so we have indeed computed the cohomology of the strata in the Maroni stratification of $\mathcal{T}_g,$ within a certain range.\\

In order to deduce the cohomology of $\mathcal{T}_g$ from that of the strata, we consider the spectral sequence associated to this stratification. Recall that 
$$
\operatorname{dim}\mathcal{N}_n=2g+2-n-\delta_{0,n},
$$
then each $\mathcal{N}_{n+2}$ has codimension $2$ in $\mathcal{N}_n$ with the sole exception of $\mathcal{N}_2\subset\mathcal{T}_g,$ which is a divisor for $g$ even. Moreover, observe from Proposition 1.1 that $N_0$ is the only stratum having different cohomology from the other strata. Consequently, we will need to distinguish the cases for $g$ even and odd.

\subsection{Case $g$ even}
Suppose first that $g$ is even. We can recover the rational cohomology of $\mathcal{T}_g$, in a certain range, from the Gysin spectral sequence in Borel-Moore homology induced by the Maroni stratification \eqref{Mar}. 

Precisely, the $E^1\mbox{-}$ page of the spectral sequence is obtained by considering in each column the Borel-Moore homology of each strata $N_n$ in \eqref{Mar}, twisted by $\mathbf{Q}(-\operatorname{codim}_{\mathcal{T}_g}N_n).$
We will also twist the whole spectral sequence by $\mathbf{Q}(-\operatorname{dim}\mathcal{T}_g)$ in order to get the fundamental class of $\mathcal{T}_g$ in degree 0. The spectral sequence is represented in Table 4.

	\begin{table}[H]\caption{Spectral sequence converging to $\bar{H}_\bullet(\mathcal{T}_g;\mathbf{Q})\otimes\mathbf{Q}(-\operatorname{dim}\mathcal{T}_g)$ with $g$ even.}\small\label{ss1}\centering
		\begin{tabular}{ ccccccc|c} 
		 & & &$N_6$ &$N_4$&$N_2$&$N_0$&\\
			 &$\dots$& $-4$& $-3$&$-2$&$-1$&$0$&\\
			\hline
			 & & & &&&$\mathbf{Q}$&$0$\\
			 & & & &&$\mathbf{Q}(-1)$&&$-1$\\
			 && & & &&&$-2$\\
			 && & & &$\mathbf{Q}(-2)$&& $-3$\\
			 && & & $\mathbf{Q}(-3)$\tikzmark{11}&&&$-4$\\
			 && &&&&\tikzmark{1}$\mathbf{Q}(-3)$&$-5$\\
			 &&&&$\mathbf{Q}(-4)$\tikzmark{22}&\tikzmark{2}$\mathbf{Q}(-4)$ & &$-6$ \\
			&& &$\mathbf{Q}(-5)$ \tikzmark{33}&& & & $-7$\\
			 && & & &\tikzmark{3}$\mathbf{Q}(-5)$&&$-8$\\
			 && &$\mathbf{Q}(-6)$\tikzmark{44} &\tikzmark{4}$\mathbf{Q}(-6)$ &&&$-9$\\
			 &&$\mathbf{Q}(-7)$\tikzmark{55} & & &&&$-10$\\
			&& & &\tikzmark{5}$\mathbf{Q}(-7)$ &&&$-11$\\
			 &&$\dots$ &$\mathbf{Q}(-8)$ & &&&$-12$\\
			 &&$\dots$ &$\mathbf{Q}(-9)$ & &&&$-13$\\
		\end{tabular}
		\begin{tikzpicture}[overlay, remember picture, yshift=.25\baselineskip, shorten >=.5pt, shorten <=.5pt]
		\draw [shorten <=.1cm,->]([yshift=3pt]{pic cs:1}) -- ([yshift=3pt]{pic cs:11});
		\draw [shorten <=.1cm,->]([yshift=3pt]{pic cs:2}) -- ([yshift=3pt]{pic cs:22});
		\draw [shorten <=.1cm,->]([yshift=3pt]{pic cs:3}) -- ([yshift=3pt]{pic cs:33});
		\draw [shorten <=.1cm,->]([yshift=3pt]{pic cs:4}) -- ([yshift=3pt]{pic cs:44});
		\draw [shorten <=.1cm,->]([yshift=3pt]{pic cs:5}) -- ([yshift=3pt]{pic cs:55});
		\end{tikzpicture}
	\end{table}

	Consider the differentials highlighted in Table 4. To check whether they have rank 1 or not, it suffices to study the fundamental classes of each strata $\left[\mathcal{N}_n\right]$ in $\mathcal{T}_g,$ which has been already done in \cite{PV13} and \cite{PV15}.\\
	Precisely, by \cite[Theorem 3.3]{PV13}, any Chow class in $\mathcal{N}_n$ is the restriction of a tautological class on $\mathcal{M}_g$ and its fundamental class is a multiple of $(n-1)\mbox{-}$th power of $\kappa_1,$ by \cite[Proposition 6.2]{PV15}, where $\kappa_1\in R^1(\mathcal{M}_g)\subset H^2(\mathcal{M}_g;\mathbf{Q})$ and $R^*(\mathcal{M}_g)$ denotes the tautological ring of $\mathcal{M}_g.$\\
	Then, the fundamental classes $\left[\mathcal{N}_n\right]$ must vanish for $n\geq 4$ by \cite[Theorem 1.1]{CL}. This means that the differentials $d_{p,q}^2:E_{p,q}^2\rightarrow E_{p-2,q+1}^2$ must be of rank 1.\\
	Because of the ring structure, the second non-trivial cohomology class in each column is also a power of $\kappa_1$ and thus, also the differentials $d_{p,q}^1:E_{p,q}^1\rightarrow E_{p-1,q}^1$ must be of rank 1.\\
	
	Therefore we may conclude that, in degree $i<\frac{g}{4},$
	
	\begin{equation*}
	H^i(\mathcal{T}_g;\mathbf{Q})=\begin{cases}
	\mathbf{Q},&i=0;\\
	\mathbf{Q}(-1)&i=2;\\
	\mathbf{Q}(-2)&i=4;\\
	0&\text{otherwise}.
	\end{cases}
	\end{equation*}

	where the bound $i<\frac{g}{4}$ is obtained by recalling from the previous section that the cohomology of each strata $N_n$ is known in degree lower than $\frac{d-3n+1}{2},$ where $d=\frac{g+3n+2}{2}.$ For any $0\leq n\leq \lfloor{\frac{g+2}{3}}\rfloor,$ we require
	\begin{align*}
	i&<\operatorname{min}\left\{2\operatorname{codim_{\mathcal{T}_g}}N_n+\frac{d-3n+1}{2} ;\,0\leq n\leq \left\lfloor{\frac{g+2}{3}}\right\rfloor\right\}-1\\
	&=\operatorname{min}\left\{2\operatorname{max}\{0,n-1\}+\frac{g+3n+2}{4}+\frac{-3n+1}{2};\,0\leq n\leq \left\lfloor{\frac{g+2}{3}}\right\rfloor \right\}-1\\
	&=\frac{g}{4}.
	\end{align*}

\subsection{Case $g$ odd}
Consider now the odd case. The $E^1\mbox{-}$ page of the Gysin spectral sequence in Borel-Moore homology induced by the Maroni stratification (twisted again by $\mathbf{Q}(-\operatorname{dim}\mathcal{T}_g)$) is represented in Table 5.
\begin{table}[h!]\caption{Spectral sequence converging to $\bar{H}_\bullet(\mathcal{T}_g;\mathbf{Q})\otimes\mathbf{Q}(-\operatorname{dim}\mathcal{T}_g)$ with $g$ odd.}\small\label{ss2}\centering
	\begin{tabular}{ ccccccc|c} 
		& & &$N_7$ &$N_5$&$N_3$&$N_1$&\\
		&$\dots$& $-4$& $-3$&$-2$&$-1$&$0$&\\
		\hline
		& & & &&&$\mathbf{Q}$&$0$\\
		& & & &&&&$-1$\\
		&& & & &&$\mathbf{Q}(-1)$&$-2$\\
		&& & & &$\mathbf{Q}(-2)$&& $-3$\\
		&& & & &&&$-4$\\
		&& &&&$\mathbf{Q}(-3)$\tikzmark{aa}&\tikzmark{a}$\mathbf{Q}(-3)$&$-5$\\
		&&&&$\mathbf{Q}(-4)$\tikzmark{bb}& & &$-6$ \\
		&& & && & \tikzmark{b}$\mathbf{Q}(-4)$ & $-7$\\
		&& & & $\mathbf{Q}(-5)$\tikzmark{cc}&\tikzmark{c}$\mathbf{Q}(-5)$&&$-8$\\
		&& &$\mathbf{Q}(-6)$\tikzmark{dd} & &&&$-9$\\
		&& & & &\tikzmark{d}$\mathbf{Q}(-6)$&&$-10$\\
		&& &$\mathbf{Q}(-7)$\tikzmark{ee} &\tikzmark{e}$\mathbf{Q}(-7)$ &&&$-11$\\
		&&$\mathbf{Q}(-8)$\tikzmark{ff} & & &&&$-12$\\
		&&$\dots$ & &\tikzmark{f}$\mathbf{Q}(-8)$ &&&$-13$\\
		&&$\dots$ &$\dots$ & &&&$-14$\\
	\end{tabular}
	\begin{tikzpicture}[overlay, remember picture, yshift=.25\baselineskip, shorten >=.5pt, shorten <=.5pt]
		\draw [shorten <=.1cm,->]([yshift=3pt]{pic cs:a}) -- ([yshift=3pt]{pic cs:aa});
		\draw [shorten <=.1cm,->]([yshift=3pt]{pic cs:b}) -- ([yshift=3pt]{pic cs:bb});
		\draw [shorten <=.1cm,->]([yshift=3pt]{pic cs:c}) -- ([yshift=3pt]{pic cs:cc});
		\draw [shorten <=.1cm,->]([yshift=3pt]{pic cs:d}) -- ([yshift=3pt]{pic cs:dd});
		\draw [shorten <=.1cm,->]([yshift=3pt]{pic cs:e}) -- ([yshift=3pt]{pic cs:ee});
	\draw [shorten <=.1cm,->]([yshift=3pt]{pic cs:f}) -- ([yshift=3pt]{pic cs:ff});
	\end{tikzpicture}
\end{table}
	
For the same reasons discussed in the even case, the differentials highlighted in Table 5 are all of rank 1. Thus, the stable rational cohomology of $\mathcal{T}_g,$ with $g$ odd, coincides with the one obtained in the even case and precisely, in degree $i<\frac{g-3}{4}$,
	
\begin{equation*}
	H^i(\mathcal{T}_g;\mathbf{Q})=\begin{cases}
	\mathbf{Q},&i=0;\\
	\mathbf{Q}(-1)&i=2;\\
	\mathbf{Q}(-2)&i=4;\\
	0&\text{otherwise},
	\end{cases}
\end{equation*}
where $\frac{g-3}{4}=\operatorname{min}\left\{2(n-1)+\frac{g+3n+2}{4}+\frac{-3n+1}{2};\,1\leq n\leq \left\lfloor{\frac{g+2}{3}}\right\rfloor\right\}-1.$\\
Comparing both results, obtained for $g$ even and odd, 
we get that the rational cohomology $H^{\bullet}(\mathcal{T}_g;\mathbf{Q})$ stabilizes to its rational Chow ring, and equivalently to its tautological ring, for $g$ sufficiently large.\\

Finally, let us conclude by giving the proof of Corollary 1.4.
\begin{proof}[Proof of Corollary 1.4]
	Let us go back to the computation of the cohomology of the strata $N_n,$ 
	which was obtained by considering the quotient space
	$X_{d,n}/G,$
	with $G=\mathbf{C}^*\times GL_2(\mathbf{C})$ for any $n\geq 1$ and $G=\mathbf{C}^*\times SL_2(\mathbf{C})\times SL_2(\mathbf{C})$ for $n=0.$\\
	Then by taking first the projectivization $\mathbf{P}X_{d,n},$ considering then the quotients  $\mathbf{P}X_{d,n}/\mathbf{C}^*$ and $\mathbf{P}X_{d,0}/ SL_2$ would give us the rational cohomology, until a certain degree, of a $SL_2\mbox{-}$cover of $N_n$ that we will denote by $N^\dagger_n,$ for any $n\geq 0$.\\
	Consider first $n=0.$ As we have already noticed in section 4, in this case the generalized version of Leray-Hirsch theorem can be applied to the whole $G=\mathbf{C}^*\times SL_2(\mathbf{C})\times SL_2(\mathbf{C}),$ meaning that the cohomology of $X_{d,0}$ is completely divisible by that of $\mathbf{C}^*\times SL_2(\mathbf{C})\times SL_2(\mathbf{C}).$ Therefore the rational cohomology of $N_0^\dagger$ is simply that of $X_{d,0}$ divided by that of $\mathbf{C}^*\times SL_2(\mathbf{C}):$ 
	\begin{equation*}
		H^i({N}^\dagger_0;\mathbf{Q})=\begin{cases}
			\mathbf{Q},&i=0;\\
			\mathbf{Q}(-2)&i=3;\\
			\mathbf{Q}(-3)&i=5;\\
			\mathbf{Q}(-5)&i=8;\\
			0&\text{otherwise};
		\end{cases}
	\end{equation*}
	in degree $i\leq\left\lfloor\frac{d}{2}\right\rfloor.$ \\
	For $n\geq1,$ we want to consider the quotient of $\mathbf{P}X_{d,n}$ by $\mathbf{C}^*,$ and in this case we saw that the Leray-Hirsch theorem applies only to $GL_2(\mathbf{C}).$ Hence, the cohomology of $N^\dagger_n$ is obtained by multiplying that of $N_n$ by that of $GL_2(\mathbf{C}),$ and then dividing by cohomology of the copy of $\mathbf{C}^*,$ contained in $GL_2(\mathbf{C}):$
	\begin{equation*}
		H^i({N}^\dagger_n;\mathbf{Q})=\begin{cases}
			\mathbf{Q},&i=0;\\
			\mathbf{Q}(-1)&i=2;\\
			\mathbf{Q}(-2)&i=3;\\
			2\mathbf{Q}(-3)&i=5;\\
			\mathbf{Q}(-4)&i=7;\\
			\mathbf{Q}(-5)&i=8;\\
			\mathbf{Q}(-6)&i=10;\\
			0&\text{otherwise}.
		\end{cases}
	\end{equation*}
	in degree $i\leq\left\lfloor\frac{d-3n}{2}\right\rfloor.$
	Now, by looking at the Maroni stratification, all these $N^\dagger_n$ fit together into a moduli space which will be a $SL_2\mbox{-}$cover of $\mathcal{T}_g,$ which we will denote by $\mathcal{T}^\dagger_g,$ and whose cohomology can be deduced by writing the analogues of Tables 4 and 5.\\
	When $g$ is even,
	\begin{table}[H]\caption{Spectral sequence converging to $\bar{H}_\bullet(\mathcal{T}^\dagger_g;\mathbf{Q})\otimes\mathbf{Q}(-\operatorname{dim}\mathcal{T}^\dagger_g)$ with $g$ even.}\small\label{blabla1}\centering
		\begin{tabular}{ ccccccc|c} 
			& & &$N^\dagger _6$ &$N^\dagger_4$&$N^\dagger_2$&$N^\dagger_0$&\\
			&$\dots$& $-4$& $-3$&$-2$&$-1$&$0$&\\
			\hline
			& & & &&&$\mathbf{Q}$&$0$\\
			& & & &&$\mathbf{Q}(-1)$&&$-1$\\
			&& & & &&&$-2$\\
			&& & & &$\mathbf{Q}(-2)$&$\mathbf{Q}(-2)$& $-3$\\
			&& & &$\mathbf{Q}(-3)$ &$\mathbf{Q}(-3)$&&$-4$\\
			&& &&&&$\mathbf{Q}(-3)$&$-5$\\
			&&&&$\mathbf{Q}(-4)$&2$\mathbf{Q}(-4)$ & &$-6$ \\
			&&  &$\mathbf{Q}(-5)$&$\mathbf{Q}(-5)$ &  && $-7$\\
			& & & &&$\mathbf{Q}(-5)$&$\mathbf{Q}(-5)$&$-8$\\
			& & &$\mathbf{Q}(-6)$ &2$\mathbf{Q}(-6)$&$\mathbf{Q}(-6)$&&$-9$\\
			& &$\mathbf{Q}(-7)$ & $\mathbf{Q}(-7)$&&&&$-10$\\
			& && &$\mathbf{Q}(-7)$&$\mathbf{Q}(-7)$&&$-11$\\
			&&$\mathbf{Q}(-8)$ &2$\mathbf{Q}(-8)$ &$\mathbf{Q}(-8)$&&&$-12$\\
			&&$\dots$ &$\dots$ &$\dots$&$\dots$&&$-13$\\	
		\end{tabular}
	\end{table}
\noindent
On the other hand, if $g$ is odd,
\begin{table}[H]\caption{Spectral sequence converging to $\bar{H}_\bullet(\mathcal{T}^\dagger_g;\mathbf{Q})\otimes\mathbf{Q}(-\operatorname{dim}\mathcal{T}^\dagger_g)$ with $g$ odd.}\small\label{blabla}\centering
	\begin{tabular}{ ccccccc|c} 
		& & &$N^\dagger_7$ &$N^\dagger_5$&$N^\dagger_3$&$N^\dagger_1$&\\
		&$\dots$& $-4$& $-3$&$-2$&$-1$&$0$&\\
		\hline
		& & & &&&$\mathbf{Q}$&$0$\\
		& & & &&&&$-1$\\
		&& & & &&$\mathbf{Q}(-1)$&$-2$\\
		&& & & &$\mathbf{Q}(-2)$&$\mathbf{Q}(-2)$& $-3$\\
		&& & & &&&$-4$\\
		&& &&&$\mathbf{Q}(-3)$&2$\mathbf{Q}(-3)$&$-5$\\
		&&&&$\mathbf{Q}(-4)$&$\mathbf{Q}(-4)$ & &$-6$ \\
		&&  && &  &$\mathbf{Q}(-4)$& $-7$\\
		& & & &$\mathbf{Q}(-5)$&2$\mathbf{Q}(-5)$&$\mathbf{Q}(-5)$&$-8$\\
		& & &$\mathbf{Q}(-6)$ &$\mathbf{Q}(-6)$&&&$-9$\\
		& & & &&$\mathbf{Q}(-6)$&$\mathbf{Q}(-6)$&$-10$\\
		& &&$\mathbf{Q}(-7)$ &2$\mathbf{Q}(-7)$&$\mathbf{Q}(-7)$&&$-11$\\
		&&$\mathbf{Q}(-8)$ &$\mathbf{Q}(-8)$ &&&&$-12$\\
		&&$\dots$ &$\dots$ &$\dots$&$\dots$&&$-13$\\
	\end{tabular}
\end{table}
\noindent
In both spectral sequences, all the differentials between the columns, under the third row, must have rank 1 for the same reasons discussed in Tables 4 and 5. Also, the differential $E^1_{0,-3}\rightarrow E^1_{-1,-3}$ must be of rank 1, from \cite[Prop. 6.1]{PV15} and \cite[Vistoli's Theorem]{PV15}. Here in fact, Patel and Vakil proved that the Chow ring of $\mathcal{T}_{g}^\dagger$ is generated by the tautological class $\kappa_1$ and it is related to that of $\mathcal{T}_g$ by 
$$A^\bullet(\mathcal{H}_{3,g}^\dagger)=A^\bullet(\mathcal{T}_g)/(\mu),$$
where $\mu$ is a multiple of $\kappa_1^2.$ 
Therefore, for both $g$ even and odd, we have, in degree $i<\left\lfloor\frac{g}{4}\right\rfloor,$
\begin{equation*}
	H^i(\mathcal{T}^\dagger_g;\mathbf{Q})=\begin{cases}
		\mathbf{Q},&i=0;\\
		\mathbf{Q}(-1)&i=2;\\
		\mathbf{Q}(-3)&i=5;\\
		\mathbf{Q}(-4)&i=7;\\
		0&\text{otherwise}.
	\end{cases}
\end{equation*}
\end{proof}
\bibliographystyle{alpha}
\bibliography{StableRef}

\end{document}